\documentclass[a4paper,11pt,leqno]{article}
\usepackage{latexsym, amsthm, amsfonts, amsmath}
\newtheorem{theorem}{Theorem}[section]
\newtheorem{lemma}[theorem]{Lemma}
\newtheorem{e-proposition}[theorem]{Proposition}

\newtheorem{definition}[theorem]{Definition\rm}
\newtheorem{remark}{\it Remark\/}

\newtheorem{example}{\it Example\/}
\begin{document}
\title{On the geometry of lightlike submanifolds of indefinite statistical manifolds}
\author{Varun Jain, Amrinder Pal Singh and Rakesh Kumar}
\date{}
\maketitle
\begin{abstract}
We study lightlike submanifolds of indefinite statistical manifolds. Contrary to the classical theory of submanifolds of statistical manifolds, lightlike submanifolds of indefinite statistical manifolds need not to be statistical submanifold. Therefore we obtain some conditions for a lightlike submanifold of indefinite statistical manifolds to be a lightlike statistical submanifold. We derive the expression of statistical sectional curvature and finally obtain some conditions for the induced statistical Ricci tensor on a lightlike submanifold of indefinite statistical manifolds to be symmetric.
\end{abstract}

\section{Introduction}
Information geometry uses tools of differential geometry to study statistical inference, information loss, and
estimation. In fact the set of normal distributions
$$
p(x; \mu, \sigma) = \frac{1}{\sqrt{2\pi}\sigma}e^{- \frac{(x - \mu)^{2}}{2\sigma^{2}}},\quad x\in\mathbb{R},
$$
with $(\mu, \sigma)\in\mathbb{R}\times(0, +\infty)$, can be considered as a two-dimensional surface and the amount of information between the distributions is measured by the endowed Riemannian metric, the Fisher information metric. Then the family of normal distributions $p(x; \mu, \sigma)$ becomes a space of constant negative curvature and any normal distribution can be visualized as a point in the Poincare upper-half plane. Furthermore, the notion of statistical manifolds was studied in terms of information geometry. Statistical manifolds \cite{amari1} are inspired from statistical model, where the density function, the Fisher information matrix, the skewness tensor (which measures the cummulants of third order), the dual connections $\nabla^{(-1)}$ and $\nabla^{(1)}$ are replaced by an arbitrary Riemannian manifold $\tilde{M}$, the Riemannian metric $\tilde{g}$ of $\tilde{M}$, a $3$-covariant skewness tensor, the dual connections $\tilde{\nabla}$ and $\tilde{\nabla}^{*}$, respectively. Since the geometry of statistical manifolds includes dual connections which are similar to the conjugate connections of the affine geometry, therefore the geometry of statistical manifolds is related to affine differential geometry. The geometry of statistical manifolds has significant applications in various fields of science and engineering but very limited information available.\\
\indent To fill up important missing parts in the general theory of submanifolds, Duggal and Bejancu \cite{1} introduced the notion of lightlike submanifolds of semi-Riemannian manifolds and further developed by many others, see \cite{klds} and many references therein. The geometry of lightlike submanifolds has extensive uses in mathematical physics, particularly, in general theory of relativity. Moreover, the theory of lightlike submanifolds has
interaction with some results on Killing horizon, electromagnetic and radiation fields and asymptotically flat spacetimes. Therefore the study of lightlike submanifolds is an active area of study in the geometry of submanifolds.\\
\indent Although the notion of lightlike submanifolds of semi-Riemannian manifolds is well known, the one for indefinite statistical manifolds is not yet established. In this paper, our aim is to establish the theory of lightlike submanifolds of indefinite statistical manifolds. We obtain some conditions for a lightlike submanifold of indefinite statistical manifolds to be a lightlike statistical submanifold. We derive the expression of statistical sectional curvature and finally obtain some conditions for the induced statistical Ricci tensor on a lightlike submanifold of indefinite statistical manifolds to be symmetric.
\section{Lightlike submanifolds}
In this paper, we consider smooth manifolds and $\Gamma(T\bar{M})$, $\Gamma(T\bar{M}^{(p, q)})$ means the set of all vector fields and the set of all tensor fields of the type $(p, q)$ on the smooth manifold $\bar{M}$, respectively.\\
\indent Let $(\bar{M}, \bar{g})$ be a real $(m + n)$-dimensional semi-Riemannian manifold of constant index $q$ such that $m, n \geq 1$, $1\leq q \leq m + n - 1$ and $(M, g)$ be an $m$-dimensional submanifold of $\bar{M}$ and $g$ be the induced metric of $\bar{g}$ on $M$. If $\bar{g}$ is degenerate on the tangent bundle $TM$ of $M$ then $M$ is called a lightlike submanifold of $\bar{M}$. For a degenerate metric $g$ on $M$, $T_{x}M^{\bot}$ is a degenerate $n$-dimensional subspace of $T_{x}\bar{M}$. Thus, both $T_{x}M$ and $T_{x}M^{\bot}$ are degenerate orthogonal subspaces but no longer complementary. In this case, there exists a subspace $Rad(T_{x}M) = T_{x}M \cap T_{x}M^{\bot}$, known as radical (null) subspace. If the mapping $Rad(TM) : x\in M \longrightarrow Rad(T_{x}M)$, defines a smooth distribution on $M$ of rank $r>0$ then submanifold $M$ of $\bar{M}$ is called an $r$-lightlike submanifold
and $Rad(TM)$ is called the radical distribution on $M$. Screen distribution $S(TM)$ is a semi-Riemannian complementary distribution of $Rad(TM)$ in $TM$, that is, $TM = Rad(TM) \bot S(TM)$. Let $S(TM^{\bot})$ be a complementary vector subbundle to $Rad(TM)$ in $TM^{\bot}$ which is also non-degenerate with respect to $\bar{g}$. Let $tr(TM)$ be complementary (but not orthogonal) vector bundle to $TM$ in $T\bar{M}\mid _M$ then $tr(TM) = ltr(TM) \bot S(TM^{\bot})$, where $ltr(TM)$ is complementary to $Rad(TM)$ in $S(TM^{\bot})^{\bot}$ and is an arbitrary lightlike transversal vector bundle of $M$. Thus we have $T\bar{M}\mid_{M} = TM \oplus tr(TM) = (Rad(TM) \oplus ltr(TM)) \bot S(TM) \bot S(TM^{\bot})$, (for detail see \cite{1}). Let $\mathcal{U}$ be a local coordinate neighborhood of M then local quasi-orthonormal field of frames on $\bar{M}$ along M is $\{\xi_{1},...,\xi_{r}, X_{r+1},..., X_{m}, N_{1},..., N_{r}, W_{r+1},..., W_{n}\},$ where $\{\xi_{i}\}_{i=1}^{r}$ and $\{N_{i}\}_{i=1}^{r}$ are lightlike basis of $\Gamma(Rad(TM)|_{\mathcal{U}})$ and $\Gamma(ltr(TM)|_{\mathcal{U}})$, respectively and $\{X_{\alpha}\}_{\alpha=r+1}^{m}$ and $\{W_{a}\}_{a=r+1}^{n}$ are orthonormal basis of $\Gamma(S(TM)|_{\mathcal{U}})$ and $\Gamma(S(TM^{\bot})|_{\mathcal{U}})$, respectively. These local quasi-orthonormal field of frames on $\bar{M}$ satisfy
$$
\bar{g}(N_{i}, \xi_{j}) = \delta^{i}_{j},\quad \bar{g}(N_{i}, N_{j}) = \bar{g}(N_{i}, X_{\alpha}) = \bar{g}(N_{i}, W_{a}) = 0.
$$
\noindent Let $\tilde{\nabla}$ be the Levi-Civita connection on $\bar{M}$, then Gauss and Weingarten formulae are
\begin{equation}\label{eq:1}
\tilde{\nabla}_{X}Y = \nabla_{X}Y + h(X, Y),\quad \tilde{\nabla}_{X}U = - A_{U}X + \nabla^{t}_{X}U,
\end{equation}
for $X, Y \in \Gamma(TM)$ and $U \in \Gamma(tr(TM))$, where $\{\nabla_{X}Y, A_{U}X\}$ and $\{h(X, Y), \nabla^{t}_{X} U\}$ belongs to $\Gamma(TM)$ and $\Gamma(tr(TM))$, respectively. Here $\nabla$ is a torsion-free linear connection on $M$, $h$ is a symmetric bilinear form on $TM$ which is called the second fundamental form, $A_{U}$ is a linear operator on $M$ and known as the shape operator. Considering the projection morphisms $L$ and $S$ of $tr(TM)$ on $ltr (TM)$ and $S(TM^{\bot})$, respectively, then (\ref{eq:1}) becomes
\begin{equation}\label{eq:2}
\tilde{\nabla}_{X}Y = \nabla_{X}Y + h^{l}(X, Y) + h^{s}(X, Y),\quad
\tilde{\nabla}_{X}U = - A_{U}X + D_{X}^{l}U + D_{X}^{s}U,
\end{equation}
where $h^{l}(X, Y) = L(h(X, Y))$, $h^{s}(X, Y) = S(h(X, Y))$, $D_{X}^{l}U = L(\nabla^{\bot}_{X}U)$, $D_{X}^{s}U = S(\nabla^{\bot}_{X}U)$. As $h^{l}$ and $h^{s}$ are $\Gamma(ltr (TM))$-valued and $\Gamma(S(TM^{\bot}))$-valued, respectively, therefore they are called as the lightlike second fundamental form and the screen second fundamental form on $M$. In particular, we have
\begin{equation}\label{eq:3}
\tilde{\nabla}_{X}N = - A_{N}X + \nabla_{X}^{l}N + D^{s}(X, N),\quad
\tilde{\nabla}_{X}W = - A_{W}X + \nabla_{X}^{s}W + D^{l}(X, W),
\end{equation}
where $X \in \Gamma(TM)$, $N \in \Gamma(ltr (TM))$ and $W \in \Gamma(S(TM^{\bot}))$. Then using (\ref{eq:2}) and (\ref{eq:3}), we have
\begin{equation}\label{eq:4}
\bar{g}(h^{s}(X, Y), W) + \bar{g}(Y, D^{l}(X, W)) = g(A_{W}X, Y).
\end{equation}
Let $P$ be the projection morphism of $T M$ on $S(TM)$ then
\begin{equation}\label{eq:5}
\nabla_{X}PY = \nabla^{'}_{X}PY + h^{'}(X, PY),\quad \nabla_{X}\xi = - A^{'}_{\xi}X + \nabla^{'t}_{X}\xi,
\end{equation}
where $\{\nabla^{'}_{X}PY, A^{'}_{\xi}X\}$ and $\{h^{'}(X, Y), \nabla^{'t}_{X}\xi\}$ belongs to $\Gamma(S(TM))$ and $\Gamma(Rad(TM))$, respectively. $\nabla^{'}$ and $\nabla^{'t}$ are linear connections on $S(TM)$ and $Rad(TM)$, respectively. $h^{'}$ and $A^{'}$ are $\Gamma(Rad(TM))$-valued and
$\Gamma(S(TM))$-valued bilinear forms and they are called as the second fundamental forms of distributions $S(TM)$ and $Rad(TM)$, respectively. Using (\ref{eq:2}) and (\ref{eq:5}), we obtain
\begin{equation}\label{eq:6}
\bar{g}(h^{l}(X, PY), \xi) = g(A^{'}_{\xi}X, PY),\quad \bar{g}(h^{'}(X, PY), N) = g(A_{N}X, PY),
\end{equation}
for any $X, Y\in \Gamma(TM)$, $\xi\in\Gamma(Rad(TM))$ and $N\in\Gamma(ltr(TM))$.\\
\indent From the geometry of non-degenerate submanifolds, it is known that the induced connection $\nabla$ on a non-degenerate submanifold is always a metric connection. Unfortunately, this is not true for lightlike submanifolds and particularly satisfies
\begin{equation}\label{eq:7}
(\nabla_{X}g)(Y, Z) = \bar{g}(h^{l}(X, Y), Z) + \bar{g}(h^{l}(X, Z), Y),
\end{equation}
for any $X, Y, Z \in \Gamma(TM)$.
\section{Lightlike Submanifolds of Indefinite Statistical Manifolds}
Let $(\bar{M}, \bar{g})$ be a semi-Riemannian manifold equipped with a semi-Riemannian metric $\bar{g}$ of constant index $q$ and $\bar{\nabla}$ be an affine torsion free connection on $\bar{M}$. A pair $(\bar{\nabla}, \bar{g})$ is called a statistical structure on $\bar{M}$ if $\bar{\nabla}$ is torsion free and the Codazzi equation
\begin{equation}\label{eq:8}
(\bar{\nabla}_{X}\bar{g})(Y, Z) = (\bar{\nabla}_{Y}\bar{g})(X, Z),
\end{equation}
holds for any vector fields $X, Y$ and $Z$ of $\bar{M}$. If $(\bar{\nabla}, \bar{g})$ is a statistical structure on a semi-Riemannian manifold $\bar{M}$ then the triplet $(\bar{M}, \bar{g}, \bar{\nabla})$ is called an indefinite statistical manifold. The affine connection $\bar{\nabla}^{*}$ which is also assumed to be torsion free on $(\bar{M}, \bar{g}, \bar{\nabla})$ is called the dual connection of $\bar{\nabla}$ with respect to $\bar{g}$ if it satisfies
\begin{equation}\label{eq:9}
X\bar{g}(Y, Z) = \bar{g}(\bar{\nabla}_{X}Y, Z) + \bar{g}(Y, \bar{\nabla}^{*}_{X}Z),
\end{equation}
for any vector fields $X, Y$ and $Z$ of $\bar{M}$. If $(\bar{\nabla}, \bar{g})$ is a statistical structure on $\bar{M}$ then so is $(\bar{\nabla}^{*}, \bar{g})$ and furthermore $(\bar{\nabla}^{*})^{*} = \bar{\nabla}$. Therefore now onwards we denote an indefinite statistical manifold by $(\bar{M}, \bar{g}, \bar{\nabla}, \bar{\nabla}^{*})$. Let $\bar{\nabla}^{\bar{g}}$ be the Levi-Civita connection of $\bar{g}$ then we have $\bar{\nabla}^{\bar{g}} = \frac{1}{2}(\bar{\nabla} + \bar{\nabla}^{*})$. Moreover from (\ref{eq:8}), it is clear that a semi-Riemannian manifold is always a statistical manifold and $(\bar{\nabla}^{\bar{g}})^{*} = \bar{\nabla}^{\bar{g}}$.\\
\indent In fact the notion of statistical structure comes from information geometry, for details see \cite{amari}. Let $p(\cdot, \theta): (\chi, dx)\rightarrow (0, \infty)$ be the probability density parameterized by $\theta = (\theta^{1},\ldots, \theta^{n})\in\Theta\subset\mathbb{R}^{n}$. Then for any $\alpha\in\mathbb{R}$, we have
$$
g_{\theta} = \sum\Big\{\int_{\chi}\frac{\partial\log p}{\partial\theta^{i}}(x, \theta)\frac{\partial\log p}{\partial\theta^{j}}(x, \theta)p(x, \theta)dx\Big\}d\theta^{i}d\theta^{j},
$$
and
$$
\Gamma_{ijk}^{(\alpha)}(\theta) = \int_{\chi}\Big\{\frac{\partial^{2}\log p}{\partial\theta^{i}\partial\theta^{i}}(x, \theta) + \frac{1 - \alpha}{2}\frac{\partial\log p}{\partial\theta^{i}}(x, \theta)\frac{\partial\log p}{\partial\theta^{j}}(x, \theta)\Big\}\frac{\partial\log p}{\partial\theta^{k}}(x, \theta)p(x, \theta)dx.
$$
It is easy to see that $g_{\theta}$ is a positive semi-definite quadratic form on $T_{\theta}\Theta$. If $g$ is a Riemannian metric on $\Theta$ then $(\Theta, \nabla^{(\alpha)}, g)$ is a statistical manifold, where $\nabla^{(\alpha)}$ is an affine connection, known as the Amari's $\alpha$-connection with respect to $\{p(\cdot, \theta)|\theta\in\Theta\}$ and defined by $\Gamma_{ijk}^{(\alpha)} = g(\nabla^{(\alpha)}_{\frac{\partial}{\partial\theta^{i}}}\frac{\partial}{\partial\theta^{j}}, \frac{\partial}{\partial\theta^{k}})$\\
\indent Let $(\bar{\nabla}, \bar{g})$ be a statistical structure on $\bar{M}$ then the difference tensor field $K\in\Gamma(T\bar{M}^{(1, 2)})$ is given by (see \cite{hf})
\begin{equation}\label{eq:30}
K(X, Y) = \bar{\nabla}_{X}Y - \bar{\nabla}^{\bar{g}}_{X}Y,
\end{equation}
for any $X, Y\in\Gamma(T\bar{M})$ and the difference tensor field satisfies
\begin{equation}\label{eq:31}
K(X, Y) = K(Y, X),\quad \bar{g}(K(X, Y), Z) = \bar{g}(Y, K(X, Z)).
\end{equation}
Conversely, for a Riemannian metric $\bar{g}$ if the given $K\in\Gamma(T\bar{M}^{(1, 2)})$ satisfies (\ref{eq:31}) then a pair $(\bar{\nabla} = \bar{\nabla}^{\bar{g}} + K, \bar{g})$ becomes a statistical structure on $\bar{M}$ and also
\begin{equation}\label{eq:32}
K = \bar{\nabla} - \bar{\nabla}^{\bar{g}} = \frac{1}{2}(\bar{\nabla} - \bar{\nabla}^{*}).
\end{equation}
\begin{theorem}
Let $(\bar{M}, \bar{g}, \bar{\nabla}, \bar{\nabla}^{*})$ be an indefinite statistical manifold and $\bar{\nabla}^{*}$ is a dual connection of $\bar{\nabla}$ with respect to the metric $\bar{g}$. Then
\begin{equation}\label{eq:36}
(\bar{\nabla}_{X}\bar{g})(Y, Z) + (\bar{\nabla}^{*}_{X}\bar{g})(Y, Z) = 0,
\end{equation}
for any $X, Y\in\Gamma(T\bar{M})$.
\end{theorem}
\begin{proof}
Let $X, Y\in\Gamma(T\bar{M})$ then using (\ref{eq:30}), we obtain
$$
(\bar{\nabla}_{X}\bar{g})(Y, Z) = (\bar{\nabla}^{\bar{g}}_{X}\bar{g})(Y, Z) - \bar{g}(K(X, Y), Z) - \bar{g}(Y, K(X, Z),
$$
then using the fact that $\bar{\nabla}^{\bar{g}}$ is a metric connection with (\ref{eq:31}), we get
\begin{equation}\label{eq:33}
(\bar{\nabla}_{X}\bar{g})(Y, Z) = - 2\bar{g}(K(X, Y), Z).
\end{equation}
Analogously using (\ref{eq:31}) and (\ref{eq:32}), we also obtain
\begin{equation}\label{eq:34}
(\bar{\nabla}_{X}\bar{g})(Y, Z) = (\bar{\nabla}^{*}_{X}\bar{g})(Y, Z) - 4\bar{g}(K(X, Y), Z),
\end{equation}
then using (\ref{eq:33}) in (\ref{eq:34}), we have
\begin{equation}\label{eq:35}
(\bar{\nabla}^{*}_{X}\bar{g})(Y, Z) = 2\bar{g}(K(X, Y), Z).
\end{equation}
Hence from (\ref{eq:33}) and (\ref{eq:35}), the assertion follows.
\end{proof}
\indent Let $(M, g)$ be a lightlike submanifold of an indefinite statistical manifold $(\bar{M}, \bar{g}, \bar{\nabla}, \bar{\nabla}^{*})$ and let $\nabla$, $\nabla^{*}$ be the induced linear connections on $M$ from the connections $\bar{\nabla}$, $\bar{\nabla}^{*}$, respectively. Then using geometry of lightlike submanifolds of semi-Riemannian manifolds, the Gauss and Weingarten formulas for a lightlike submanifold of an indefinite statistical manifold $(\bar{M}, \bar{g}, \bar{\nabla}, \bar{\nabla}^{*})$ are given by
\begin{equation}\label{eq:10}
\bar{\nabla}_{X}Y = \nabla_{X}Y + h^{l}(X, Y) + h^{s}(X, Y),\quad \bar{\nabla}^{*}_{X}Y = \nabla^{*}_{X}Y + h^{*l}(X, Y) + h^{*s}(X, Y),
\end{equation}
\begin{equation}\label{eq:11}
\bar{\nabla}_{X}N = - A_{N}X + \nabla^{l}_{X}N + D^{s}(X, N),\quad \bar{\nabla}^{*}_{X}N = - A^{*}_{N}X + \nabla^{*l}_{X}N + D^{*s}(X, N),
\end{equation}
\begin{equation}\label{eq:12}
\bar{\nabla}_{X}W = - A_{W}X + \nabla^{s}_{X}W + D^{l}(X, W),\quad \bar{\nabla}^{*}_{X}W = - A^{*}_{W}X + \nabla^{*s}_{X}W + D^{*l}(X, W),
\end{equation}
for any $X, Y\in\Gamma(TM)$, $N\in\Gamma(ltr(TM))$ and $W\in\Gamma(S(TM^{\bot}))$. Let $P$ be the projection morphism of $TM$ on $S(TM)$ then following (\ref{eq:5}), we have
\begin{equation}\label{eq:12a}
\nabla_{X}PY = \nabla^{'}_{X}PY + h^{'}(X, PY),\quad \nabla^{*}_{X}PY = \nabla^{*'}_{X}PY + h^{*'}(X, PY),
\end{equation}
\begin{equation}\label{eq:12b}
\nabla_{X}\xi = - A^{'}_{\xi}X + \nabla^{'t}_{X}\xi,\quad \nabla^{*}_{X}\xi = - A^{*'}_{\xi}X + \nabla^{*'t}_{X}\xi.
\end{equation}
Since $\bar{\nabla}$ is a torsion free affine connection therefore using (\ref{eq:10}) in $\bar{\nabla}_{X}Y - \bar{\nabla}_{Y}X - [X, Y] = 0$ and then on equating the tangential components, it follows that the induced connection $\nabla$ is a torsion free linear connection and analogously $\nabla^{*}$ is also a torsion free linear connection.
\begin{lemma}
Let $(M, g)$ be a lightlike submanifold of an indefinite statistical manifold $(\bar{M}, \bar{g}, \bar{\nabla}, \bar{\nabla}^{*})$ then we have
\begin{equation}\label{eq:16}
(\nabla_{X}g)(Y, Z) - (\nabla_{Y}g)(X, Z) = \bar{g}(Y, h^{l}(X, Z)) - \bar{g}(X, h^{l}(Y, Z)),
\end{equation}
\begin{equation}\label{eq:13}
g(A_{W}X, Y) - g(X, A_{W}Y) = \bar{g}(D^{l}(X, W), Y) - \bar{g}(X, D^{l}(Y, W)),
\end{equation}
\begin{equation}\label{eq:15}
\bar{g}(h^{l}(X, \xi), Y) - \bar{g}(h^{l}(Y, \xi), X) = g(X, \nabla_{Y}\xi) - g(\nabla_{X}\xi, Y),
\end{equation}
\begin{equation}\label{eq:14}
(\nabla^{l}_{X}\bar{g})(Y, N) - (\nabla^{l}_{Y}\bar{g})(X, N) = g(A_{N}Y, X) - g(A_{N}X, Y),
\end{equation}
for any $X, Y, Z\in\Gamma(TM)$, $\xi\in\Gamma(Rad(TM))$, $N\in\Gamma(ltr(TM))$ and $W\in\Gamma(S(TM^{\bot}))$, where $(\nabla^{l}_{X}\bar{g})(Y, N) = X\bar{g}(Y, N) - \bar{g}(\nabla_{X}Y, N) - \bar{g}(Y, \nabla^{l}_{X}N)$.
\end{lemma}
\begin{proof}
By straightforward calculations using (\ref{eq:10}) to (\ref{eq:12}) in (\ref{eq:8}), the Lemma follows.
\end{proof}
\noindent Analogously using the fact that $(\bar{\nabla}^{*}, \bar{g})$ is also a statistical structure, we have the following lemma immediately.
\begin{lemma}
Let $(M, g)$ be a lightlike submanifold of an indefinite statistical manifold $(\bar{M}, \bar{g}, \bar{\nabla}, \bar{\nabla}^{*})$ then we have
\begin{equation}\label{eq:17}\tag{22*}
(\nabla^{*}_{X}g)(Y, Z) - (\nabla^{*}_{Y}g)(X, Z) = \bar{g}(Y, h^{*l}(X, Z)) - \bar{g}(X, h^{*l}(Y, Z)),
\end{equation}
\begin{equation}\label{eq:13a}\tag{23*}
g(A^{*}_{W}X, Y) - g(X, A^{*}_{W}Y) = \bar{g}(D^{*l}(X, W), Y) - \bar{g}(X, D^{*l}(Y, W)),
\end{equation}
\begin{equation}\label{eq:15d}\tag{24*}
\bar{g}(h^{*l}(X, \xi), Y) - \bar{g}(h^{*l}(Y, \xi), X) = g(X, \nabla^{*}_{Y}\xi) - g(\nabla^{*}_{X}\xi, Y),
\end{equation}
\begin{equation}\label{eq:14a}\tag{25*}
(\nabla^{*l}_{X}\bar{g})(Y, N) - (\nabla^{*l}_{Y}\bar{g})(X, N) = g(A^{*}_{N}Y, X) - g(A^{*}_{N}X, Y),
\end{equation}
for any $X, Y, Z\in\Gamma(TM)$, $\xi\in\Gamma(Rad(TM))$, $N\in\Gamma(ltr(TM))$ and $W\in\Gamma(S(TM^{\bot}))$, where $(\nabla^{*l}_{X}\bar{g})(Y, N) = X\bar{g}(Y, N) - \bar{g}(\nabla^{*}_{X}Y, N) - \bar{g}(Y, \nabla^{*l}_{X}N)$.
\end{lemma}
\begin{remark}
The equation (\ref{eq:17}) is consider as the dual of (\ref{eq:16}) and similar assumptions for other equations. Now onwards we omit the dual equation and put $*$ to the number of original equation, if needed.
\end{remark}
\begin{lemma}
Let $(M, g)$ be a lightlike submanifold of an indefinite statistical manifold $(\bar{M}, \bar{g}, \bar{\nabla}, \bar{\nabla}^{*})$ then we have
\begin{equation}\label{eq:15a}
\bar{g}(h^{l}(X, Y), \xi) + \bar{g}(Y, h^{*l}(X, \xi)) + g(Y, \nabla^{*}_{X}\xi) = 0,
\end{equation}
\begin{equation}\label{eq:15b}
\bar{g}(h^{s}(X, Y), W) + \bar{g}(D^{*l}(X, W), Y) = g(A^{*}_{W}X, Y),
\end{equation}
\begin{equation}\label{eq:15c}
\bar{g}(\nabla_{X}Y, N) + \bar{g}(Y, \nabla^{*l}_{X}N) = X\bar{g}(Y, N) + g(Y, A^{*}_{N}X),
\end{equation}
\begin{equation}\label{eq:18}
Xg(Y, Z) = g(\nabla_{X}Y, Z) + g(Y, \nabla^{*}_{X}Z) + \bar{g}(h^{l}(X, Y), Z) + \bar{g}(Y, h^{*l}(X, Z)),
\end{equation}
for any $X, Y, Z\in\Gamma(TM)$, $\xi\in\Gamma(Rad(TM))$, $N\in\Gamma(ltr(TM))$ and $W\in\Gamma(S(TM^{\bot}))$.
\end{lemma}
\begin{proof}
Lemma follows by using (\ref{eq:10}) to (\ref{eq:12}) in (\ref{eq:9}).
\end{proof}
\begin{remark}
Particularly let $X, Y\in\Gamma(S(TM))$ then from (\ref{eq:13}), we have $g(A_{W}X, Y) = g(X, A_{W}Y)$, implies that $A_{W}$ is a self-adjoint operator on $S(TM)$.
\end{remark}
\noindent From (\ref{eq:16}), (\ref{eq:17}) and (\ref{eq:18}), it is obvious that the lightlike submanifolds $(M, g, \nabla)$ and $(M, g, \nabla^{*})$ of an indefinite statistical manifold $(\bar{M}, \bar{g}, \bar{\nabla}, \bar{\nabla}^{*})$ are not lightlike statistical submanifolds. Hence we have the following observation from (\ref{eq:16}), (\ref{eq:17}) and (\ref{eq:18}) immediately.
\begin{theorem}
Let $(\bar{M}, \bar{g}, \bar{\nabla}, \bar{\nabla}^{*})$ be an indefinite statistical manifold. If $h^{l}$ and $h^{*l}$ vanish identically then lightlike submanifolds $(M, g, \nabla)$ and $(M, g, \nabla^{*})$ become lightlike statistical submanifolds of indefinite statistical manifold $\bar{M}$ and $\nabla^{*}$ becomes a dual connection of $\nabla$ with respect to the induced metric $g$.
\end{theorem}
\begin{theorem}
Let $(\bar{M}, \bar{g}, \bar{\nabla}, \bar{\nabla}^{*})$ be an indefinite statistical manifold such that $\nabla^{*}$ is a dual connection of $\nabla$ with respect to the induced metric $g$. Then $(M, g, \nabla)$ is a lightlike statistical submanifold of $\bar{M}$ if and only if $(M, g, \nabla^{*})$ is a lightlike statistical submanifold of $\bar{M}$.
\end{theorem}
\begin{proof}
Suppose $(M, g, \nabla)$ is a lightlike statistical submanifold of an indefinite statistical manifold $(\bar{M}, \bar{g}, \bar{\nabla}, \bar{\nabla}^{*})$ then from (\ref{eq:16}), it follows that $\bar{g}(Y, h^{l}(X, Z)) - \bar{g}(X, h^{l}(Y, Z)) = 0,$ for any $X, Y, Z\in\Gamma(TM)$. Further using (\ref{eq:9}) and (\ref{eq:10}), we obtain
\begin{eqnarray*}
0 &=& \bar{g}(Y, h^{*l}(X, Z)) - \bar{g}(X, h^{*l}(Y, Z)) + \{Zg(X, Y) - g(\nabla_{Z}X, Y) - g(X, \nabla^{*}_{Z}Y)\}\nonumber\\&&
- \{Zg(Y, X) - g(\nabla_{Z}Y, X) - g(Y, \nabla^{*}_{Z}X)\}.
\end{eqnarray*}
On using the hypothesis that $\nabla^{*}$ is a dual connection of $\nabla$ with respect to the induced metric $g$ in above expression, it gives $\bar{g}(Y, h^{*l}(X, Z)) - \bar{g}(X, h^{*l}(Y, Z)) = 0$ and hence from (\ref{eq:17}) it follows that $(M, g, \nabla^{*})$ becomes a lightlike statistical submanifold of $\bar{M}$. Similarly we can prove the converse part.
\end{proof}
Recall that a vector field $X$ on a semi-Riemannian manifold $(\bar{M}, \bar{g})$ is said to be a Killing vector field (infinitesimal isometry) if $(L_{X}\bar{g})(Y, Z) = 0$, where
\begin{equation}\label{eq:19}
(L_{X}\bar{g})(Y, Z) = X(\bar{g}(Y, Z)) - \bar{g}([X, Y], Z) - \bar{g}(Y, [X, Z]),
\end{equation}
for any vector fields $Y$ and $Z$ on $\bar{M}$. Further a distribution $\mathcal{D}$ on $\bar{M}$ is said to be a Killing distribution if each vector field belonging to $\mathcal{D}$ is a Killing vector field. Let $X, Y, Z\in\Gamma(TM)$ then using (\ref{eq:10}) in (\ref{eq:19}), we obtain
$$
(L_{X}\bar{g})(Y, Z) = (\nabla_{X}g)(Y, Z) + g(\nabla_{Y}X, Z) + g(Y, \nabla_{Z}X).
$$
Interchange the role of $X$ and $Y$ in above equation and take $Z = \xi$, then on subtracting the resulting equation from it, we derive
$$
(L_{X}\bar{g})(Y, \xi) - (L_{Y}\bar{g})(X, \xi) = g(\nabla_{\xi}X, Y) - g(\nabla_{\xi}Y, X).
$$
\begin{theorem}
Let $(M, g, \nabla)$ be a lightlike statistical submanifold of an indefinite statistical manifold $(\bar{M}, \bar{g}, \bar{\nabla}, \bar{\nabla}^{*})$. The the following statements are equivalent:
\begin{itemize}
  \item[(a)] The radical distribution $Rad(TM)$ is a Killing distribution.
  \item[(b)] The radical distribution $Rad(TM)$ is a parallel distribution with respect to the induced connection $\nabla$, that is, $\nabla_{X}\xi\in\Gamma(Rad(TM))$, for any $X\in\Gamma(TM)$ and $\xi\in\Gamma(Rad(TM))$.
  \item[(c)] $A^{'}_{\xi}$ vanishes of $\Gamma(TM)$ for any $\xi\in\Gamma(Rad(TM))$.
\end{itemize}
\end{theorem}
\begin{proof}
Let $(M, g, \nabla)$ be a lightlike statistical submanifold of an indefinite statistical manifold $(\bar{M}, \bar{g}, \bar{\nabla}, \bar{\nabla}^{*})$ then
\begin{equation}\label{eq:20}
(\nabla_{X}g)(Y, Z) = (\nabla_{Y}g)(X, Z),
\end{equation}
for any $X, Y, Z\in\Gamma(TM)$. Then from (\ref{eq:16}), we have $\bar{g}(Y, h^{l}(X, Z)) = \bar{g}(X, h^{l}(Y, Z))$ and on taking $Z = \xi$ it follows that $\bar{g}(h^{l}(X, \xi), Y) = \bar{g}(h^{l}(Y, \xi), X)$. Using this result in (\ref{eq:15}), we further obtain
\begin{equation}\label{eq:21}
g(\nabla_{X}\xi, Y) = g(\nabla_{Y}\xi, X),
\end{equation}
for any $X, Y\in\Gamma(TM)$ and $\xi\in\Gamma(Rad(TM))$.\\
\indent Next for any $X, Y\in\Gamma(TM)$ and $\xi\in\Gamma(Rad(TM))$, using (\ref{eq:10}) in (\ref{eq:19}), we derive $(L_{\xi}g)(X, Y) = \{\xi g(X, Y) - g(\nabla_{\xi}X, Y) - g(X, \nabla_{\xi}Y)\} + \{g(\nabla_{X}\xi, Y) + g(\nabla_{Y}\xi, X)\}$ and further using (\ref{eq:20}) and (\ref{eq:21}), it implies that
\begin{eqnarray}\label{eq:22}
&(L_{\xi}g)(X, Y)& = (\nabla_{\xi}g)(X, Y) + 2g(\nabla_{X}\xi, Y)\nonumber\\&&
= (\nabla_{X}g)(\xi, Y) + 2g(\nabla_{X}\xi, Y)\nonumber\\&&
= g(\nabla_{X}\xi, Y).
\end{eqnarray}
On using (\ref{eq:12b}) in (\ref{eq:22}), it gives that
\begin{equation}\label{eq:23}
(L_{\xi}g)(X, Y) = g(\nabla_{X}\xi, Y) = - g(A^{'}_{\xi}X, Y).
\end{equation}
Thus the assertion follows from (\ref{eq:23})
\end{proof}
\begin{theorem}
Let $(M, g, \nabla)$ be a lightlike statistical submanifold of an indefinite statistical manifold $(\bar{M}, \bar{g}, \bar{\nabla}, \bar{\nabla}^{*})$. The the following statements are equivalent:
\begin{itemize}
  \item[(a)] The screen distribution $S(TM)$ is integrable.
  \item[(b)] $A_{N}$ is self adjoint operator on $S(TM)$ with respect to $g$.
  \item[(c)] $h^{'}$ is symmetric on $S(TM)$.
\end{itemize}
\end{theorem}
\begin{proof}
Let $X, Y\in\Gamma(S(TM))$ and $N\in\Gamma(ltr(TM))$ then from (\ref{eq:8}), we have $(\bar{\nabla}_{X}\bar{g})(Y, N) = (\bar{\nabla}_{Y}\bar{g})(X, N)$, using (\ref{eq:10}) and the fact that $\nabla$ is torsion free, it follows that
\begin{equation}\label{eq:24}
\bar{g}([X, Y], N) = g(A_{N}X, Y) - g(X, A_{N}Y).
\end{equation}
Further, since $\nabla$ is torsion free therefore for any $X, Y\in\Gamma(S(TM))$ and $N\in\Gamma(ltr(TM))$, using (\ref{eq:12a}), we have
\begin{eqnarray}\label{eq:25}
\bar{g}([X, Y], N) = \bar{g}(\nabla_{X}Y, N) - \bar{g}(\nabla_{Y}X, N) = \bar{g}(h^{'}(X, Y) - h^{'}(Y, X), N).
\end{eqnarray}
Hence assertions follows from (\ref{eq:24}) and (\ref{eq:25}).
\end{proof}
\begin{theorem}
Let $(M, g, \nabla)$ be a lightlike statistical submanifold of an indefinite statistical manifold $(\bar{M}, \bar{g}, \bar{\nabla}, \bar{\nabla}^{*}
)$. Then the radical distribution $Rad(TM)$ is always integrable.
\end{theorem}
\begin{proof}
Let $(M, g, \nabla)$ be a lightlike statistical submanifold of an indefinite statistical manifold $(\bar{M}, \bar{g}, \bar{\nabla}, \bar{\nabla}^{*})$ then
for any $\xi_{1}, \xi_{2}\in\Gamma(Rad(TM))$ and $X\in\Gamma(S(TM))$ from (\ref{eq:20}), we obtain $(\nabla_{\xi_{1}}g)(\xi_{2}, X) = (\nabla_{\xi_{2}}g)(\xi_{1}, X)$. This implies that $g(\nabla_{\xi_{1}}\xi_{2}, X) = g(\nabla_{\xi_{2}}\xi_{1}, X)$, that is, $g([\xi_{1}, \xi_{2}], X) = 0$ and hence the proof is complete.
\end{proof}
\indent Let $M$ be a lightlike submanifold of an indefinite statistical manifold $(\bar{M}, \bar{g}, \bar{\nabla}, \bar{\nabla}^{*})$. Let $\bar{R}$, $R$, $R^{l}$ and $R^{s}$ be the curvature tensor fields of $\bar{\nabla}$, $\nabla$, $\nabla^{l}$ and $\nabla^{s}$ respectively. Then using (\ref{eq:10}) to (\ref{eq:12}) for any $X, Y, Z\in\Gamma(TM)$, $N\in\Gamma(ltr(TM))$ and $W\in\Gamma(S(TM^{\bot}))$, we have the following observation (see also \cite{1}):
\begin{theorem}\label{thm:A}
Let $M$ be a lightlike submanifold of an indefinite statistical manifold $(\bar{M}, \bar{g}, \bar{\nabla}, \bar{\nabla}^{*})$ then following equations of Gauss, Codazzi and Ricci hold
\begin{eqnarray}\label{eq:26}
(\bar{R}(X, Y)Z)^{tangential} &=& R(X, Y)Z + A_{h^{l}(X, Z)}Y - A_{h^{l}(Y, Z)}X + A_{h^{s}(X, Z)}Y\nonumber\\&& - A_{h^{s}(Y, Z)}X,
\end{eqnarray}
\begin{eqnarray}\label{eq:26a}
(\bar{R}(X, Y)Z)^{transversal} &=& (\nabla_{X}h^{l})(Y, Z) - (\nabla_{Y}h^{l})(X, Z) + (\nabla_{X}h^{s})(Y, Z)\nonumber\\&&
- (\nabla_{Y}h^{s})(X, Z) + D^{l}(X, h^{s}(Y, Z)) - D^{l}(Y, h^{s}(X, Z))\nonumber\\&&
+ D^{s}(X, h^{l}(Y, Z)) - D^{s}(Y, h^{l}(X, Z)),
\end{eqnarray}
\begin{eqnarray}\label{eq:27}
(\bar{R}(X, Y)N)^{tangential} &=& (\nabla_{Y}A)(N, X) - (\nabla_{X}A)(N, Y) + A_{D^{s}(X, N)}Y \nonumber\\&&
- A_{D^{s}(Y, N)}X,
\end{eqnarray}
\begin{eqnarray}\label{eq:27a}
(\bar{R}(X, Y)N)^{transversal} &=& R^{l}(X, Y)N + h^{l}(Y, A_{N}X) -  h^{l}(X, A_{N}Y) + h^{s}(Y, A_{N}X)\nonumber\\&&
-  h^{s}(X, A_{N}Y) + (\nabla_{X}D^{s})(Y, N) - (\nabla_{Y}D^{s})(X, N)\nonumber\\&&
+ D^{l}(X, D^{s}(Y, N)) - D^{l}(Y, D^{s}(X, N)),
\end{eqnarray}
\begin{eqnarray}\label{eq:28}
(\bar{R}(X, Y)W)^{tangential} &=& (\nabla_{Y}A)(W, X) - (\nabla_{X}A)(W, Y) + A_{D^{l}(X, W)}Y\nonumber\\&&
- A_{D^{l}(Y, W)}X,
\end{eqnarray}
\begin{eqnarray}\label{eq:28a}
(\bar{R}(X, Y)W)^{transversal} &=& R^{s}(X, Y)W + h^{l}(Y, A_{W}X) -  h^{l}(X, A_{W}Y) + h^{s}(Y, A_{W}X)\nonumber\\&&
-  h^{s}(X, A_{W}Y) + (\nabla_{X}D^{l})(Y, W) - (\nabla_{Y}D^{l})(X, W)\nonumber\\&&
+ D^{s}(X, D^{l}(Y, W)) - D^{s}(Y, D^{l}(X, W)),
\end{eqnarray}
where
\begin{eqnarray}
&&(\nabla_{X}h^{l})(Y, Z) = \nabla^{l}_{X}(h^{l}(Y, Z)) - h^{l}(\nabla_{X}Y, Z) - h^{l}(Y, \nabla_{X}Z),\nonumber\\&&
(\nabla_{X}h^{s})(Y, Z) = \nabla^{s}_{X}(h^{s}(Y, Z)) - h^{s}(\nabla_{X}Y, Z) - h^{s}(Y, \nabla_{X}Z),\nonumber\\&&
(\nabla_{X}A)(N, Y) = \nabla_{X}(A(N, Y)) - A(\nabla^{l}_{X}N, Y) - A(N, \nabla_{X}Y),\nonumber\\&&
(\nabla_{X}D^{l})(Y, W) = \nabla^{l}_{X}(D^{l}(Y, W)) - D^{l}(\nabla_{X}Y, W) - D^{l}(Y, \nabla^{s}_{X}W),\nonumber\\&&
(\nabla_{X}A)(W, Y) = \nabla_{X}(A(W, Y)) - A(\nabla^{s}_{X}W, Y) - A(W, \nabla_{X}Y),\nonumber\\&&
(\nabla_{X}D^{s})(Y, N) = \nabla^{s}_{X}(D^{s}(Y, N)) - D^{s}(\nabla_{X}Y, N) - D^{s}(Y, \nabla^{l}_{X}N).\nonumber
\end{eqnarray}
\end{theorem}
Recall  that a submanifold $(M, g)$ of a Riemannian manifold $(\bar{M}, \bar{g})$ is called \emph{totally geodesic} if any geodesic on the submanifold $M$ with its induced Riemannian metric $g$ is also a geodesic on the Riemannian manifold $(\bar{M}, \bar{g})$ and moreover $M$ is totally geodesic in $\bar{M}$ if and only if the second fundamental form on $M$ vanishes identically. Therefore Duggal and Jin \cite{kldj} defined that a lightlike submanifold $M$ of a semi-Riemannian manifold $(\bar{M}, \bar{g})$ with the Levi-civita connection $\bar{\nabla}$ is a totally geodesic lightlike submanifold if $h^{l}(X, Y) = h^{s}(X, Y) = 0$, for all $X, Y\in\Gamma(TM)$. Now, let $(\bar{M}, \bar{g}, \bar{\nabla}, \bar{\nabla}^{*})$ be an indefinite statistical manifold with affine connections $\bar{\nabla}$ and $\bar{\nabla}^{*}$ then a lightlike submanifold $M$ of $\bar{M}$ is said to be $\bar{\nabla}$-autoparallel (respectively, $\bar{\nabla}^{*}$-autoparallel) if $(\bar{\nabla}_{X}Y)_{p} = ((\bar{\nabla}|_{M})_{X}Y)_{p}$ (respectively, $(\bar{\nabla}^{*}_{X}Y)_{p} = ((\bar{\nabla}^{*}|_{M})_{X}Y)_{p}$, for any $p\in M$ and $X, Y\in\Gamma(TM)$, that is, if $h^{l}(X, Y) = h^{s}(X, Y) = 0$ (respectively, $h^{*l}(X, Y) = h^{*s}(X, Y) = 0$), for all $X, Y\in\Gamma(TM)$. The submanifold $M$ is called dual-autoparallel if $M$ is both $\bar{\nabla}$- and $\bar{\nabla}^{*}$-autoparallel, that is, if $h^{l}(X, Y) = h^{*l}(X, Y) = 0$ and $h^{s}(X, Y) = h^{*s}(X, Y) = 0$, for all $X, Y\in\Gamma(TM)$, see \cite{cu}.
\begin{remark}
Let $M$ be $\bar{\nabla}$-autoparallel (respectively, $\bar{\nabla}^{*}$-autoparallel) lightlike submanifold of an indefinite statistical $(\bar{M}, \bar{g}, \bar{\nabla}, \bar{\nabla}^{*})$ then using (\ref{eq:16}) (respectively, (\ref{eq:17})), $(M, \nabla, g)$ (respectively, $(M, \nabla^{*}, g)$) becomes a lightlike statistical submanifold of $\bar{M}$. Further, if $M$ is dual-autoparallel then using (\ref{eq:16}), (\ref{eq:17}) and (\ref{eq:18}), both $(M, \nabla, g)$ and $(M, \nabla^{*}, g)$ become lightlike statistical submanifolds of $\bar{M}$ and $\nabla^{*}$ becomes a dual connection of $\nabla$ with respect to the induced metric $g$.
\end{remark}
\begin{theorem}\label{thm:B}
Let $M$ be $\bar{\nabla}$-autoparallel (respectively, $\bar{\nabla}^{*}$-autoparallel) lightlike submanifold of an indefinite statistical $(\bar{M}, \bar{g}, \bar{\nabla}, \bar{\nabla}^{*})$ then $\bar{R}(X, Y)Z = R(X, Y)Z$, (respectively, $\bar{R}^{*}(X, Y)Z = R^{*}(X, Y)Z)$, for any $X, Y\in\Gamma(TM)$.
\end{theorem}
\noindent Using (\ref{eq:9}) for $X, Y, Z, U\in\Gamma(TM)$, we get $\bar{g}(\bar{\nabla}_{X}h^{s}(Y, Z), U) = - \bar{g}(h^{s}(Y, Z), h^{*s}(X, U))$ and $\bar{g}(\bar{\nabla}^{*}_{Y}h^{*s}(X, U), Z) = - \bar{g}(h^{s}(Y, Z), h^{*s}(X, U))$ therefore it implies that
\begin{equation}\label{eq:29}
\bar{g}(\bar{\nabla}_{X}h^{s}(Y, Z), U) = \bar{g}(\bar{\nabla}^{*}_{Y}h^{*s}(X, U), Z).
\end{equation}
Then from (\ref{eq:29}), we have the following observation.
\begin{lemma}
Let $(M, g)$ be a lightlike submanifold of an indefinite statistical manifold $(\bar{M}, \bar{g}, \bar{\nabla}, \bar{\nabla}^{*})$. If $M$ is $\bar{\nabla}$-autoparallel then $h^{*s}$ is parallel with respect to $\bar{\nabla}^{*}$ and if $M$ is $\bar{\nabla}^{*}$-autoparallel then $h^{s}$ is parallel with respect to $\bar{\nabla}$.
\end{lemma}
\begin{definition}
A statistical structure $(\bar{\nabla}, \bar{g})$ is said to be of constant curvature $k\in\mathbb{R}$ if $\bar{R}(X, Y)Z = k\{\bar{g}(Y, Z)X - \bar{g}(X, Z)Y\}$ holds for any vector fields $X$, $Y$ and $Z$ on $\bar{M}$. A statistical structure $(\bar{\nabla}, \bar{g})$ of constant curvature $0$ is called a Hessian structure.
\end{definition}
\noindent Let $K\in\Gamma(T\bar{M}^{(1, 2)})$ be the difference tensor field for a statistical structure $(\bar{\nabla}, \bar{g})$ then $Q = - \bar{\nabla}K\in\Gamma(T\bar{M}^{(1, 3)})$ is called the Hessian curvature tensor \cite{shima}. If for $c\in\mathbb{R}$, $(\bar{\nabla}_{X}K)(Y, Z) = - \frac{c}{2}\{\bar{g}(X, Y)Z + \bar{g}(X, Z)Y\}$, for any $X, Y, Z\in\Gamma(T\bar{M})$ then the Hessian structure is said to be of constant Hessian curvature $c$.
\begin{example}\cite{hf}
Let $(H, \tilde{g})$ be the upper half space of constant curvature $- 1$ then
$$
H = \{y = (y^{1},\ldots, y^{n+1})\in\mathbb{R}^{n+1}| y^{n+1} > 0\},
$$
$$
\tilde{g} = (y^{n+1})^{- 2}\sum_{A=1}^{n+1}dy^{A}dy^{A}.
$$
If an affine connection $\tilde{\nabla}$ on $H$ given by
$$
\tilde{\nabla}_{\frac{\partial}{\partial y^{n+1}}}\frac{\partial}{\partial y^{n+1}} = (y^{n+1})^{- 1}\frac{\partial}{\partial y^{n+1}},\quad
\tilde{\nabla}_{\frac{\partial}{\partial y^{i}}}\frac{\partial}{\partial y^{j}} = 2\delta_{ij}(y^{n+1})^{-1}\frac{\partial}{\partial y^{n+1}},
$$
$$
\tilde{\nabla}_{\frac{\partial}{\partial y^{i}}}\frac{\partial}{\partial y^{n+1}} = \tilde{\nabla}_{\frac{\partial}{\partial y^{n+1}}}\frac{\partial}{\partial y^{j}} = 0,
$$
where $i, j=1,\ldots, n$. Then $(H, \tilde{\nabla}, \tilde{g})$ becomes a Hessian manifold of constant Hessian curvature $4$. Moreover $(H, \tilde{\nabla}, \tilde{g})$ expresses the statistical model of normal distribution when dim$H$ = 2 and the normal distribution with mean $\mu$ and variance $\sigma^{2}$ is given by
$$
N(x, \mu, \sigma^{2}) = \frac{1}{\sqrt{2\pi\sigma^{2}}}\exp\Big\{- \frac{1}{2\sigma^{2}}(x - \mu)^{2}\Big\}, x,\mu\in\mathbb{R}, \sigma > 0.
$$
\end{example}
\begin{definition}(\cite{hf1}) For a statistical manifold $(\bar{M}, \bar{\nabla}, \bar{g})$, the statistical curvature tensor field $\bar{S}\in\Gamma(T\bar{M}^{(1, 3)})$ of $(\bar{M}, \bar{\nabla}, \bar{g})$ is given by
\begin{equation}\label{eq:37}
\bar{S}(X, Y)Z = \frac{1}{2}\{\bar{R}(X, Y)Z + \bar{R}^{*}(X, Y)Z\},
\end{equation}
for any $X, Y, Z\in\Gamma(T\bar{M})$. Further a statistical manifold $(\bar{M}, \bar{\nabla}, \bar{g})$ is said to be of constant sectional curvature $c\in\mathbb{R}$ if
\begin{equation}\label{eq:42}
\bar{S}(X, Y)Z = c\{\bar{g}(Y, Z)X - \bar{g}(X, Z)Y\},
\end{equation}
holds for any $X, Y, Z\in\Gamma(T\bar{M})$.
\end{definition}
\noindent The statistical curvature tensor field $S$ satisfies the following identities
\begin{equation}\label{eq:38}
\bar{g}(\bar{S}(U, Z)Y, X) = - \bar{g}(\bar{S}(Z, U)Y, X),\quad \bar{g}(\bar{S}(U, Z)Y, X) = - \bar{g}(\bar{S}(U, Z)X, Y),
\end{equation}
\begin{equation}\label{eq:39}
\bar{g}(\bar{S}(X, Y)Z, U) = \bar{g}(\bar{S}(U, Z)Y, X),\quad \bar{S}(X, Y)Z + \bar{S}(Y, Z)X + \bar{S}(Z, X)Y = 0,
\end{equation}
for any $X, Y, Z, U\in\Gamma(T\bar{M})$.\\
\indent Let $S$ be the induced statistical curvature tensor field induced on lightlike submanifold $M$ of an indefinite statistical manifold $(\bar{M}, \bar{g}, \bar{\nabla}, \bar{\nabla}^{*})$ and given by $S(X, Y)Z = \frac{1}{2}\{R(X, Y)Z + R^{*}(X, Y)Z\}$, for any $X, Y Z\in\Gamma(TM)$. Then using the expressions of the Theorem \ref{thm:A} and their dual, we have the following expressions for the statistical curvature tensor field on $\bar{M}$:
\begin{eqnarray}\label{eq:40}
2(\bar{S}(X, Y)Z)^{tangential} &=& 2S(X, Y)Z + A_{h^{l}(X, Z)}Y - A_{h^{l}(Y, Z)}X + A_{h^{s}(X, Z)}Y\nonumber\\&&
- A_{h^{s}(Y, Z)}X + A^{*}_{h^{*l}(X, Z)}Y - A^{*}_{h^{*l}(Y, Z)}X + A^{*}_{h^{*s}(X, Z)}Y\nonumber\\&&
- A^{*}_{h^{*s}(Y, Z)}X,
\end{eqnarray}
\begin{eqnarray}\label{eq:41}
2(\bar{S}(X, Y)Z)^{transversal} &=& (\nabla_{X}h^{l})(Y, Z) - (\nabla_{Y}h^{l})(X, Z) + (\nabla_{X}h^{s})(Y, Z)\nonumber\\&&
- (\nabla_{Y}h^{s})(X, Z) + D^{l}(X, h^{s}(Y, Z)) - D^{l}(Y, h^{s}(X, Z))\nonumber\\&&
+ D^{s}(X, h^{l}(Y, Z)) - D^{s}(Y, h^{l}(X, Z)) + (\nabla^{*}_{X}h^{*l})(Y, Z)\nonumber\\&&
- (\nabla^{*}_{Y}h^{*l})(X, Z) + (\nabla^{*}_{X}h^{*s})(Y, Z) - (\nabla^{*}_{Y}h^{*s})(X, Z)\nonumber\\&&
+ D^{*l}(X, h^{*s}(Y, Z)) - D^{*l}(Y, h^{*s}(X, Z))\nonumber\\&&
+ D^{*s}(X, h^{*l}(Y, Z)) - D^{*s}(Y, h^{*l}(X, Z)),
\end{eqnarray}
for $X, Y, Z\in\Gamma(TM)$ and analogously we can write expressions for $(\bar{S}(X, Y)N)^{tangential}$, $(\bar{S}(X, Y)N)^{transversal}$, $(\bar{S}(X, Y)W)^{tangential}$ and $(\bar{S}(X, Y)W)^{transversal}$.\\
\indent Next, consider the frames field $\{\xi_{1},...,\xi_{r}, e_{r+1},..., e_{m}, N_{1},..., N_{r}, W_{r+1},..., W_{n}\},$ on $\bar{M}$ along M, where $\{\xi_{i}\}_{i=1}^{r}$ and $\{N_{i}\}_{i=1}^{r}$ are lightlike bases of $\Gamma(Rad(TM)|_{\mathcal{U}})$ and $\Gamma(ltr(TM)|_{\mathcal{U}})$, respectively and $\{e_{\alpha}\}_{\alpha=r+1}^{m}$ and $\{W_{\beta}\}_{\beta=r+1}^{n}$ are orthonormal bases of $\Gamma(S(TM)|_{\mathcal{U}})$ and $\Gamma(S(TM^{\bot})|_{\mathcal{U}})$, respectively. Then the statistical Ricci tensor $\bar{R}ic$ on $\bar{M}$ is defined by
$$
\bar{R}ic(X, Y) = trace\{Z\rightarrow \bar{S}(X, Z)Y\},
$$
for any $X, Y\in\Gamma(T\bar{M})$ and locally $\bar{R}ic$ on $\bar{M}$ is given by
$$
\bar{R}ic(X, Y) = \sum_{i=1}^{m+n}\epsilon_{i}\bar{g}(\bar{S}(X, e_{i})Y, e_{i}),
$$
where $\epsilon_{i}$ is signature of $e_{i}$. Form (\ref{eq:38}) and (\ref{eq:39}), it is clear that the Ricci tensor $\bar{R}ic$ on $\bar{M}$ is symmetric. For the induced statistical curvature tensor field $S$, the induced statistical Ricci tensor $Ric$ on $M$ is defined as
$$
Ric(X, Y) = trace\{Z\rightarrow S(X, Z)Y\},
$$
for any $X, Y\in\Gamma(M)$ and locally $Ric$ on $M$ is given by
\begin{eqnarray}\label{eq:43}
Ric(X, Y) = \sum_{i=1}^{r}\bar{g}(S(X, \xi_{i})Y, N_{i}) + \sum_{\alpha=r+1}^{m}\bar{g}(S(X, e_{\alpha})Y, e_{\alpha}).
\end{eqnarray}
It should be noted that the induced Ricci tensor of a lightlike submanifold $M$ of a semi-Riemannian manifold $\bar{M}$ is not symmetric because the induced connection $\nabla$ on a lightlike submanifold $M$ is not a metric connection. Therefore induced Ricci tensor is just a tensor quantity and has no geometric or physical meaning contrary to the symmetric Ricci tensor $\bar{R}ic$ of $\bar{M}$. We have analogous situation for a lightlike submanifold $M$ of an indefinite statistical manifold $(\bar{M}, \bar{g}, \bar{\nabla}, \bar{\nabla}^{*})$, where the lightlike submanifold $M$ need not be statistical and moreover the induced metric $g$ is not metric but the statistical Ricci tensor $\bar{R}ic$ on $\bar{M}$ is symmetric. Since the symmetry of induced statistical Ricci tensor $Ric$ is important and we have following important
observation from the Theorem \ref{thm:B} immediately.
\begin{theorem}
Let $M$ be a dual-autoparallel lightlike submanifold of an indefinite statistical manifold $(\bar{M}, \bar{g}, \bar{\nabla}, \bar{\nabla}^{*})$. Then the induced statistical Ricci tensor $Ric$ on $M$ is symmetric.
\end{theorem}
\begin{definition}(\cite{kup})
A lightlike submanifold $M$ of a semi-Riemannian manifold $\bar{M}$ is said to be irrotational lightlike submanifold if $\tilde{\nabla}_{X}\xi\in\Gamma(TM)$, for any $X\in\Gamma(TM)$ and $\xi\in\Gamma(Rad(TM))$.
\end{definition}
\noindent Hence it is clear that if $M$ is an irrotational lightlike submanifold then $h^{l}(X, \xi) = h^{s}(X, \xi) = 0$, for any $X\in\Gamma(TM)$ and $\xi\in\Gamma(Rad(TM))$.
\begin{theorem}\label{thm:C}
Let $M$ be an irrotational lightlike submanifold of an indefinite statistical manifold $(\bar{M}, \bar{g}, \bar{\nabla}, \bar{\nabla}^{*})$ of constant sectional curvature $c$ such that the screen distribution $S(TM)$ is parallel with respect to the induced connection $\nabla$ and $D^{l}(X, W) = 0$, for any $X\in\Gamma(TM)$ and $W\in\Gamma(S(TM^{\bot}))$. Then the induced statistical Ricci tensor $Ric$ is symmetric.
\end{theorem}
\begin{proof}
Let $\{\xi_{i}\}_{i=1}^{r}$ and $\{N_{i}\}_{i=1}^{r}$ be lightlike bases of $\Gamma(Rad(TM)|_{\mathcal{U}})$ and $\Gamma(ltr(TM)|_{\mathcal{U}})$, respectively then using (\ref{eq:42}) and (\ref{eq:40}) for any $X, Y\in\Gamma(TM)$, we obtain
\begin{eqnarray}\label{eq:44}
\sum_{i=1}^{r}\bar{g}(S(X, \xi_{i})Y, N_{i}) &=& - rc\bar{g}(X, Y) -\frac{1}{2}\sum_{i=1}^{r}\{\bar{g}(A_{h^{l}(X, Y)}\xi_{i}, N_{i}) - \bar{g}(A_{h^{l}(\xi_{i}, Y)}X, N_{i})\nonumber\\&&
+ \bar{g}(A_{h^{s}(X, Y)}\xi_{i}, N_{i}) - \bar{g}(A_{h^{s}(\xi_{i}, Y)}X, N_{i})\nonumber\\&&
+ \bar{g}(A^{*}_{h^{*l}(X, Y)}\xi_{i}, N_{i}) - \bar{g}(A^{*}_{h^{*l}(\xi_{i}, Y)}X, N_{i})\nonumber\\&&
+ \bar{g}(A^{*}_{h^{*s}(X, Y)}\xi_{i}, N_{i}) - \bar{g}(A^{*}_{h^{*s}(\xi_{i}, Y)}X, N_{i})\}.
\end{eqnarray}
On putting $Y = \xi_{i}\in\Gamma(Rad(TM))$ in (\ref{eq:15a}), it implies that $\bar{g}(h^{l}(X, \xi_{i}) + h^{*l}(X, \xi_{i}), \xi_{i}) = 0$, it further implies $h^{l}(X, \xi_{i}) = - h^{*l}(X, \xi_{i})$. Since $M$ is an irrotational lightlike submanifold therefore $h^{l}(X, \xi_{i}) = h^{*l}(X, \xi_{i}) = 0$. Also on putting $Y = \xi_{i}$ in the dual of (\ref{eq:15b}) and then using $D^{l}(X, W) = 0$, we obtain $\bar{g}(h^{*s}(X, \xi_{i}), W) = 0$ then further non-degeneracy of $S(TM^{\bot})$ implies that $h^{*s}(X, \xi_{i}) = 0$ and moreover $h^{s}(X, \xi_{i}) = 0$, as $M$ is an irrotational lightlike submanifold. On using these facts in (\ref{eq:44}), we obtain
\begin{eqnarray}\label{eq:45}
\sum_{i=1}^{r}\bar{g}(S(X, \xi_{i})Y, N_{i}) &=& - rc\bar{g}(X, Y) -\frac{1}{2}\sum_{i=1}^{r}\{\bar{g}(A_{h^{l}(X, Y)}\xi_{i}, N_{i}) + \bar{g}(A_{h^{s}(X, Y)}\xi_{i}, N_{i})\nonumber\\&&
+ \bar{g}(A^{*}_{h^{*l}(X, Y)}\xi_{i}, N_{i}) + \bar{g}(A^{*}_{h^{*s}(X, Y)}\xi_{i}, N_{i})\}.
\end{eqnarray}
Let $\{e_{\alpha}\}_{\alpha=r+1}^{m}$ be an orthonormal basis of $\Gamma(S(TM))$ then using (\ref{eq:42}) and (\ref{eq:40}) for any $X, Y\in\Gamma(TM)$, we obtain
\begin{eqnarray}\label{eq:46}
\sum_{\alpha=r+1}^{m}\bar{g}(S(X, e_{\alpha})Y, e_{\alpha}) &=& c\sum_{\alpha=r+1}^{m}\bar{g}(e_{\alpha}, Y)\bar{g}(X, e_{\alpha}) - \bar{g}(X, Y)(m - r - 1)\nonumber\\&&
-\frac{1}{2}\sum_{\alpha=r+1}^{m}\{\bar{g}(A_{h^{l}(X, Y)}e_{\alpha}, e_{\alpha}) - \bar{g}(A_{h^{l}(e_{\alpha}, Y)}X, e_{\alpha})\nonumber\\&&
+ \bar{g}(A_{h^{s}(X, Y)}e_{\alpha}, e_{\alpha}) - \bar{g}(A_{h^{s}(e_{\alpha}, Y)}X, e_{\alpha})\nonumber\\&&
+ \bar{g}(A^{*}_{h^{*l}(X, Y)}e_{\alpha}, e_{\alpha}) - \bar{g}(A^{*}_{h^{*l}(e_{\alpha}, Y)}X, e_{\alpha})\nonumber\\&&
+ \bar{g}(A^{*}_{h^{*s}(X, Y)}e_{\alpha}, e_{\alpha}) - \bar{g}(A^{*}_{h^{*s}(e_{\alpha}, Y)}X, e_{\alpha})\}.
\end{eqnarray}
On putting $Y = e_{\alpha}$ in (\ref{eq:15b}) and its dual, we obtain
\begin{equation}\label{eq:47}
\bar{g}(A^{*}_{W}X, e_{\alpha}) = \bar{g}(h^{s}(X, e_{\alpha}), W),\quad \bar{g}(A_{W}X, e_{\alpha}) = \bar{g}(h^{*s}(X, e_{\alpha}), W),
\end{equation}
for any $X\in\Gamma(TM)$ and $W\in\Gamma(S(TM^{\bot}))$. Further using (\ref{eq:47}) in (\ref{eq:46}), we derive
\begin{eqnarray}\label{eq:48}
\sum_{\alpha=r+1}^{m}\bar{g}(S(X, e_{\alpha})Y, e_{\alpha}) &=& c\sum_{\alpha=r+1}^{m}\bar{g}(e_{\alpha}, Y)\bar{g}(X, e_{\alpha}) - \bar{g}(X, Y)(m - r - 1)\nonumber\\&&
-\frac{1}{2}\sum_{\alpha=r+1}^{m}\{\bar{g}(A_{h^{l}(X, Y)}e_{\alpha}, e_{\alpha}) - \bar{g}(A_{h^{l}(e_{\alpha}, Y)}X, e_{\alpha})\nonumber\\&&
+ \bar{g}(A_{h^{s}(X, Y)}e_{\alpha}, e_{\alpha}) - \bar{g}(h^{*s}(X, e_{\alpha}), h^{s}(e_{\alpha}, Y))\nonumber\\&&
+ \bar{g}(A^{*}_{h^{*l}(X, Y)}e_{\alpha}, e_{\alpha}) - \bar{g}(A^{*}_{h^{*l}(e_{\alpha}, Y)}X, e_{\alpha})\nonumber\\&&
+ \bar{g}(A^{*}_{h^{*s}(X, Y)}e_{\alpha}, e_{\alpha}) - \bar{g}(h^{s}(X, e_{\alpha}), h^{*s}(e_{\alpha}, Y))\}.
\end{eqnarray}
Now assume that the screen distribution be parallel with respect to the induced connection $\nabla$ therefore $\nabla_{X}e_{\alpha}\in\Gamma(S(TM))$. On taking $Y = Z = e_{\alpha}$ in (\ref{eq:18}) and then taking summation from $\alpha = r + 1$ to $m$, we get
\begin{eqnarray}\label{eq:48a}
\sum_{\alpha=r+1}^{m}\{\bar{g}(\nabla_{X}e_{\alpha}, e_{\alpha}) + \bar{g}(\nabla^{*}_{X}e_{\alpha}, e_{\alpha})\} = 0.
\end{eqnarray}
Since $\nabla_{X}e_{\alpha}\in\Gamma(S(TM))$ then using the non-degeneracy of the screen distribution, we have $\bar{g}(\nabla_{X}e_{\alpha}, e_{\alpha})\neq 0$ and on using this fact in the last expression, we get
$\bar{g}(\nabla^{*}_{X}e_{\alpha}, e_{\alpha})\neq 0$ which implies $\nabla^{*}_{X}e_{\alpha}\in\Gamma(S(TM))$. Using (\ref{eq:9}) and (\ref{eq:11}), we derive
\begin{eqnarray}\label{eq:48b}
&\bar{g}(A_{h^{l}(e_{\alpha}, Y)}X, e_{\alpha})& = - \bar{g}(\bar{\nabla}_{X}h^{l}(e_{\alpha}, Y), e_{\alpha}) = \bar{g}(h^{l}(e_{\alpha}, Y), \bar{\nabla}^{*}_{X}e_{\alpha})\nonumber\\&& = \bar{g}(h^{l}(e_{\alpha}, Y), \nabla^{*}_{X}e_{\alpha}).
\end{eqnarray}
Since $\nabla^{*}_{X}e_{\alpha}\in\Gamma(S(TM))$ then the last expression implies $\bar{g}(A_{h^{l}(e_{\alpha}, Y)}X, e_{\alpha}) = 0$. Analogously, on using the duality, we have $\bar{g}(A^{*}_{h^{*l}(e_{\alpha}, Y)}X, e_{\alpha}) = 0$. Using these facts in (\ref{eq:48}), we obtain
\begin{eqnarray}\label{eq:49}
\sum_{\alpha=r+1}^{m}\bar{g}(S(X, e_{\alpha})Y, e_{\alpha}) &=& c\sum_{\alpha=r+1}^{m}\bar{g}(e_{\alpha}, Y)\bar{g}(X, e_{\alpha}) - \bar{g}(X, Y)(m - r - 1)\nonumber\\&&
-\frac{1}{2}\sum_{\alpha=r+1}^{m}\{\bar{g}(A_{h^{l}(X, Y)}e_{\alpha}, e_{\alpha}) + \bar{g}(A_{h^{s}(X, Y)}e_{\alpha}, e_{\alpha})\nonumber\\&&
- \bar{g}(h^{*s}(X, e_{\alpha}), h^{s}(e_{\alpha}, Y)) + \bar{g}(A^{*}_{h^{*l}(X, Y)}e_{\alpha}, e_{\alpha}) \nonumber\\&&
+ \bar{g}(A^{*}_{h^{*s}(X, Y)}e_{\alpha}, e_{\alpha}) - \bar{g}(h^{s}(X, e_{\alpha}), h^{*s}(e_{\alpha}, Y))\}.
\end{eqnarray}
Thus from (\ref{eq:45}) and (\ref{eq:49}), it is immediate that the induced statistical Ricci tensor $Ric$ on $M$ is symmetric.
\end{proof}
\begin{theorem}
Let $M$ be a lightlike submanifold of an indefinite statistical manifold $(\bar{M}, \bar{g}, \bar{\nabla}, \bar{\nabla}^{*})$ of constant sectional curvature $c$ such that $D^{s}(X, N) = 0$, for any $X\in\Gamma(TM)$ and $N\in\Gamma(ltr(TM))$. Suppose there exists a transversal vector bundle of $M$ which is parallel along $TM$ with respect to the connection $\bar{\nabla}$ on $\bar{M}$, that is,
\begin{eqnarray}\label{eq:50}
\bar{\nabla}_{X}V\in\Gamma(tr(TM)),\quad \forall~ V\in\Gamma(tr(TM)),~ X\in\Gamma(TM).
\end{eqnarray}
Then the induced statistical Ricci tensor $Ric$ is symmetric.
\end{theorem}
\begin{proof}
Let $N\in\Gamma(ltr(TM))\subset\Gamma(tr(TM))$ and $W\in\Gamma(S(TM^{\bot}))\subset\Gamma(tr(TM))$ then using the hypothesis (\ref{eq:50}) in (\ref{eq:11}) and (\ref{eq:12}), we obtain respectively
\begin{equation}\label{eq:51}
A_{N}X = 0,\quad A_{W}X = 0.
\end{equation}
On putting $Y = N\in\Gamma(ltr(TM))$ and $Z = W\in\Gamma(S(TM^{\bot}))$ in (\ref{eq:9}) and then using (\ref{eq:51}), we get $\bar{g}(A^{*}_{W}X, N) = \bar{g}(D^{s}(X, N), W)$, further using the hypothesis $D^{s}(X, N) = 0$, it follows that $\bar{g}(A^{*}_{W}X, N) = 0$. This implies that
\begin{equation}\label{eq:52}
A^{*}_{W}X\in\Gamma(S(TM)).
\end{equation}
On putting $Y = Z = N\in\Gamma(ltr(TM))$ in (\ref{eq:9}) and then using (\ref{eq:51}), we get $\bar{g}(A^{*}_{N}X, N) = 0$, this implies that
\begin{equation}\label{eq:53}
A^{*}_{N}X\in\Gamma(S(TM)).
\end{equation}
On using (\ref{eq:51}) to (\ref{eq:53}) in (\ref{eq:44}), we obtain
\begin{equation}\label{eq:54}
\sum_{i=1}^{r}\bar{g}(S(X, \xi_{i})Y, N_{i}) = - rc\bar{g}(X, Y).
\end{equation}
Let $\{e_{\alpha}\}_{\alpha=r+1}^{m}$ be an orthonormal basis of $\Gamma(S(TM))$ then using (\ref{eq:9}), (\ref{eq:10}), (\ref{eq:11}) and (\ref{eq:51}), we have $0 = \bar{g}(A_{N}X, e_{\alpha}) = \bar{g}(N, \nabla^{*}_{X}e_{\alpha})$, this implies that $\nabla^{*}_{X}e_{\alpha}\in\Gamma(S(TM))$. Then following the proof of the Theorem \ref{thm:C} with (\ref{eq:48a}), it implies that $\nabla_{X}e_{\alpha}\in\Gamma(S(TM))$ using this fact, after putting $Y = e_{\alpha}$ in (\ref{eq:15c}), we obtain
\begin{equation}\label{eq:55}
\bar{g}(A^{*}_{N}X, e_{\alpha}) = 0.
\end{equation}
Now, let $X, Y\in\Gamma(TM)$ then using (\ref{eq:9}), (\ref{eq:10}) and (\ref{eq:12}), we have $\bar{g}(A_{h^{s}(e_{\alpha}, X)}Y, e_{\alpha}) = \bar{g}(h^{s}(e_{\alpha}, X), h^{*s}(Y, e_{\alpha}))$ therefore using (\ref{eq:51}), we obtain
\begin{equation}\label{eq:56}
\bar{g}(h^{s}(e_{\alpha}, X), h^{*s}(Y, e_{\alpha})) = 0,\quad\forall~X, Y\in\Gamma(TM).
\end{equation}
Also using (\ref{eq:9}), (\ref{eq:10}) and (\ref{eq:12}), we have $\bar{g}(A^{*}_{h^{*s}(e_{\alpha}, Y)}X, e_{\alpha}) = \bar{g}(h^{*s}(e_{\alpha}, Y), h^{s}(X, e_{\alpha}))$ then using (\ref{eq:56}), it follows that
\begin{equation}\label{eq:57}
\bar{g}(A^{*}_{h^{*s}(e_{\alpha}, Y)}X, e_{\alpha}) = 0.
\end{equation}
On using (\ref{eq:51}), (\ref{eq:55}) and (\ref{eq:57}) in (\ref{eq:46}), we derive
\begin{equation}\label{eq:58}
\sum_{\alpha=r+1}^{m}\bar{g}(S(X, e_{\alpha})Y, e_{\alpha}) = c\sum_{\alpha=r+1}^{m}\bar{g}(e_{\alpha}, Y)\bar{g}(X, e_{\alpha}) - \bar{g}(X, Y)(m - r - 1).
\end{equation}
Thus from (\ref{eq:54}) and (\ref{eq:58}), our assertion follows.
\end{proof}

\noindent\textbf{Author's Address}\\
Varun Jain\\
Department of Mathematics, Multani Mal Modi College, Patiala-147001, India.\\
E-mail: varun82jain@gmail.com\\[0.2cm]
Amrinder Pal Singh and Rakesh Kumar\\
Department of Basic and Applied Sciences, Punjabi University, Patiala-147002, Punjab, India.\\
Email: amrinderpalsinghramsingh@yahoo.com; rakesh$\_$bas@pbi.ac.in
\end{document}